\documentclass[11pt]{amsart}

\usepackage{amssymb}
\usepackage{eucal}
\usepackage{color}
\usepackage{ mathrsfs}
\usepackage{dsfont}

\DeclareSymbolFont{SY}{U}{psy}{m}{n}
\DeclareMathSymbol{\emptyset}{\mathord}{SY}{'306}

\theoremstyle{plain}

\newtheorem{thm}{Theorem}[section]
\newtheorem{cor}[thm]{Corollary}
\newtheorem{lem}[thm]{Lemma}
\newtheorem{prop}[thm]{Proposition}
\theoremstyle{definition}
\newtheorem{defn}[thm]{Definition}
\newtheorem{rem}[thm]{Remark}

\numberwithin{equation}{section}

\def\g{\gamma}

\def\wi{\widetilde}

\def\beq{\begin{eqnarray}}
\def\eeq{\end{eqnarray}}
\def\beqa{\begin{eqnarray*}}
\def\eeqa{\end{eqnarray*}}

\begin{document}
\title{Trace formulae for curvature of Jet Bundles over planar domain}
\author[Keshari]{Dinesh Kumar Keshari}
\address[Keshari]{Department of Mathematics, Texas A\&M University,
College Station, TX -77843, USA} \email{kesharideepak@gmail.com}
\thanks{ The work of author  was supported by IISc Research Associate
Fellowship at the Indian Institute of Science.}

\begin{abstract}
For a domain $\Omega$ in $\mathbb{C}$ and an operator $T$ in
$\mathcal{B}_n(\Omega)$, Cowen and Douglas construct a Hermitian
holomorphic vector bundle $E_T$ over $\Omega$ corresponding to
$T$.
The Hermitian holomorphic vector bundle $E_T$ is obtained as
a pull-back of the tautological bundle $S(n,\mathcal{H})$ defined
over $\mathcal{G}r(n,\mathcal{H})$ by a nondegenerate holomorphic
map $z\mapsto {\rm{ker}}(T-z),\;z\in\Omega$. To find the answer to the converse, Cowen
and Douglas studied the jet bundle in their foundational paper.
 The computations in this paper for the
curvature of the jet bundle are somewhat difficult to comprehend.
They have given a set of invariants to determine if two rank $n$
Hermitian holomorphic vector bundle are equivalent. These
invariants are complicated and not easy to compute. It is natural
to expect that the equivalence of Hermitian holomorphic jet
bundles should be easier to characterize. In fact, in the case of
the Hermitian holomorphic jet bundle
$\mathcal{J}_k(\mathcal{L}_f)$, we have shown that the curvature of
the line bundle $\mathcal{L}_f$ completely determines the class of
$\mathcal{J}_k(\mathcal{L}_f)$. In case of rank $n$ Hermitian Holomorphic
vector bundle $E_f$, We have calculated the curvature of jet bundle
$\mathcal{J}_k(E_f)$ and also have generalized the trace formula for
jet bundle $\mathcal{J}_k(E_f)$.
\end{abstract}

\maketitle
\section{Introduction}
Let $\mathcal H$ be a complex separable Hilbert space and $\mathcal
L(\mathcal H)$ denote the collection of bounded linear operators on
$\mathcal H$. The following important class of operators was introduced
in \cite{cd}.

\begin{defn}
For a connected open subset $\Omega$ of $\mathbb C$ and a
positive
integer $n$, let
\begin{eqnarray*}
\mathcal{B}_n(\Omega)  =  &\big \{& T\in\mathcal L(\mathcal H)\,|\,\,
\Omega\subset\sigma(T),\\
&& {\mathrm{ran}}\,(T-w)= \mathcal H\mbox{ for }w\in\Omega, \\
&&\bigvee_{w\in\Omega}\ker(T-w)= \mathcal H,\\
&&\dim~\ker(T-w)= n\mbox{ for } w\in\Omega\,\,  \big \},
\end{eqnarray*}
where $\sigma(T)$ denotes the spectrum of the operator $T$.
\end{defn}

We recall (cf. \cite{cd}) that an operator $T$ in the class $\mathcal{B}_n(\Omega)$
defines a Hermitian holomorphic vector bundle $E_T$ in a natural manner. It is
the sub-bundle of the trivial bundle $\Omega\times\mathcal H$ defined by
$$E_T= \{(w, x)\in\Omega\times\mathcal H: x\in \ker(T-w)\}$$
with the natural projection map $\pi:E_T\to \Omega$,  $\pi(w, x)=
w$. It is shown in \cite[Proposition 1.12]{cd} that the mapping
$w\longrightarrow \ker(T-w)$ defines a rank $n$
Hermitian holomorphic vector bundle $E_T$ over $\Omega$ for $T\in
\mathcal{B}_n(\Omega)$. In \cite{cd}, it was also shown that the equivalence
class of the Hermitian holomorphic vector bundle $E_T$ and the
unitary equivalence class of the operator $T$ determine each
other.

\begin{thm} \label{eq}
The operators $T$ and $\wi T$ in $\mathcal{B}_n(\Omega)$ are unitarily equivalent
if and only if the corresponding Hermitian holomorphic vector bundles $E_T$ and
$E_{\wi T}$ are equivalent.
\end{thm}

In general, it is not easy to decide if two  Hermitian holomorphic vector
bundles are equivalent except when the rank of the bundle is $1$. In this
case, the curvature
\begin{eqnarray*}\mathcal
K(w)= - \frac{~\partial^2\log\parallel{\gamma(w)}\parallel^2}{\partial{w}
\partial{\overline{w}}},
\end{eqnarray*}
of the line bundle $E$, defined with respect to a non-zero holomorphic section
$\g$ of $E$, is a complete invariant. The definition of the curvature is
independent of the choice of the section $\g$:  If $\g_0$ is another holomorphic
section of $E$, then
$\g_0=\phi\g$ for some holomorphic function $\phi$ on some open subset
$\Omega_0$ of $\Omega$, consequently the harmonicity of log$|\phi|$ completes
the verification.

For a domain $\Omega$ in $\mathbb{C}$ and an operator $T$ in
$\mathcal{B}_n(\Omega)$, the Hermitian holomorphic vector bundle
$E_T$ is obtained as a pull-back of the tautological bundle
$S(n,\mathcal{H})$ defined over $\mathcal{G}r(n,\mathcal{H})$ by a
nondegenerate holomorphic map $z\mapsto
{\rm{ker}}(T-z),\;z\in\Omega$ as in Definition \ref{jet:inf
non:def}. To find the answer to the converse, namely, when a given
Hermitian holomorphic vector bundle is a pull-back of the
tautological bundle by a nondegenerate holomorphic map, Cowen and
Douglas studied the jet bundle in their foundational paper
\cite[pp. 235]{cd}. The computations in this paper for the
curvature of the jet bundle are somewhat difficult to comprehend.
They have given a set of invariants to determine if two rank $n$
Hermitian holomorphic vector bundle are equivalent. These
invariants are complicated and not easy to compute. It is natural
to expect that the equivalence of Hermitian holomorphic jet
bundles should be easier to characterize. In fact, in the case of
the Hermitian holomorphic jet bundle
$\mathcal{J}_k(\mathcal{L}_f)$, where the line bundle
$\mathcal{L}_f$ is a pull-back of the tautological bundle on
$\mathcal{G}r(1,\mathcal{H})$, we have shown that the curvature of
the line bundle $\mathcal{L}_f$ completely determines the class of
$\mathcal{J}_k(\mathcal{L}_f)$. In general, however, our results
are not as complete. Relating the complex geometric invariants
inherent in the short exact sequence
\begin{eqnarray}
0\to E_{I}\to E \to E_{II}\to 0.
\end{eqnarray}
is an important problem. In the paper \cite{bottchern}, it is
shown that the Chern classes of these bundles must satisfy
$$c(E)=c(E_I)\,c(E_{II}).$$ Donaldson \cite{donald} obtains similar
relations involving what are known as secondary invariants. We
obtain a refinement, in case $E_I= \mathcal{J}_k(E_f)$ and
$E=\mathcal{J}_{k+1}(E_f)$, namely,

$$\big({\rm{trace}}\otimes
{\rm{Id}}_{n\times n}\big)(\mathcal{K}_{\mathcal{J}_k(E_f)})-
\big({\rm{trace}}\otimes {\rm{Id}}_{n\times
n}\big)(\mathcal{K}_{\mathcal{J}_{k-1}(E_f)})=\mathcal{K}_{{\mathcal{J}_k(E_f)}/{\mathcal{J}_{k-1}(E_f)}}.$$


\section{Definitions and Notations}

Here we  give the definition of a jet bundle closely following
\cite{cd}. An equivalent description, in a slightly different
language, may be found in \cite{mw}.

 Let $E$ be a Hermitian holomorphic bundle of rank $n$
over a bounded domain $\Omega\subset\mathbb{C}$. For each
$k=0,1,\ldots$ we associate to $E$ a $(k+1)n$ -dimensional
holomorphic bundle $\mathcal{J}_k(E)$, the holomorphic k-jet
bundle of $E$, defined as follows:

If $\sigma=\{\sigma_1,\ldots,\sigma_n\}$ is a holomorphic frame
for $E$, on an open subset $U$ contained in $\Omega$, then
$\mathcal{J}_k(E)$ has an associated frame
$$\mathcal{J}_k(\sigma)=\{\sigma_{10},\ldots,\sigma_{n0},\ldots,\sigma_{1k},\ldots,\sigma_{nk}\}$$
defined on $U$. If $\widetilde{\sigma}$ is another frame for $E$
defined on $\widetilde{U}$, then on $U\cap\widetilde{U}$, we have
$\widetilde{\sigma}_j=\sum a_{ij}\sigma_i$, where $A=(a_{ij})$ is
a holomorphic, $n\times n$, nonsingular matrix. Symbolically
$$\widetilde{\sigma}= \sigma A.$$ \smallskip Let
$\mathcal{J}_k(A)$ be the $(k+1)n\times (k+1)n$, non singular,
holomorphic matrix

 $$\mathcal{J}_k(A)=\left(%
\begin{array}{ccccc}
  A & A^{\prime} & A^{\prime\prime} & \cdots & \tbinom{k}{k} A^{(k)}\\
  \vdots & A & 2A^{\prime} & \cdots & \tbinom{k}{k-1}A^{(k-1)} \\
  \vdots & {} & A & \cdots & \tbinom {k}{k-2}A^{(k-2)} \\
  \vdots & {} & {} & \ddots & \vdots \\
  0 & \cdots & \cdots & \cdots & A \\
\end{array}%
\right).$$
\smallskip
 Then, by definition, the frames
$\mathcal{J}_k(\sigma)$ and $\mathcal{J}_k(\widetilde{\sigma})$
are related on $U\cap \widetilde{U}$ by
$$\mathcal{J}_k(\widetilde{\sigma})=\mathcal{J}_k(\sigma)\mathcal{J}_k(A).$$
A straightforward computation yields that if $A$ and
$\widetilde{A}$ are holomorphic $n\times n$ matrices, then
$$\mathcal{J}_k(A\widetilde{A})=\mathcal{J}_k(A)\mathcal{J}_k(\widetilde{A})$$
so the bundle $\mathcal{J}_k(E)$ is well-defined.

 The Hermitian metric $h$ on $E$ induces a Hermitian form
$\mathcal{J}_k(h)$ on $\mathcal{J}_k(E)$ such that if $h(\sigma)$
is the matrix of inner products
$\big(\!\big(\langle\sigma_j,\sigma_i\rangle\big)\!\big)_{i,j=1}^n$,
then
$$\mathcal{J}_k(h)(\mathcal{J}_k(\sigma))=\left(%
\begin{array}{ccc}
 h(\sigma) & \cdots & \frac{\partial^k h(\sigma)}{\partial z^k} \\
  \vdots & {} & \vdots \\
 \frac{\partial^k h(\sigma)}{\partial \bar{z}^k} &
 \cdots & \frac{\partial^{2k} h(\sigma)}{\partial z^k\partial \bar{z}^k} \\
\end{array}%
\right)$$ is the matrix of $\mathcal{J}_k(h)$ relative to the
frame $\mathcal{J}_k(\sigma)$. To see that $\mathcal{J}_k(h)$ is
well-defined, we need
$$\mathcal{J}_k(h)(\mathcal{J}_k(\widetilde{\sigma}))=\mathcal{J}_k(A)^*
\{\mathcal{J}_k(h)(\mathcal{J}_k(\sigma))\}\mathcal{J}_k(A)$$
which follows from the computation: For $0\leq l_1,l_2\leq k$
\begin{eqnarray}\label{wd1}\frac{\partial^{(l_1+l_2)}}{\partial z^{l_1}\partial
\bar{z}^{l_2}}h(\tilde\sigma)
=\sum_{i=1}^{l_1}\sum_{j=1}^{l_2}\binom{l_1}{i}\binom{l_2}{j}\frac{\partial^j}{\partial
\bar{z}^j}A^*\frac{\partial^{l_2+i-j}}{\partial
\bar{z}^{l_2-j}\partial
z^i}h(\sigma)\frac{\partial^{l_1-i}}{\partial
z^{l_1-i}}A.\end{eqnarray} Using equation (\ref{wd1}), we have
$$\mathcal{J}_k(h)(\mathcal{J}_k(\widetilde{\sigma}))=\mathcal{J}_k(A)^*
\{\mathcal{J}_k(h)(\mathcal{J}_k(\sigma))\}\mathcal{J}_k(A).$$
 In general, the form  $\mathcal{J}_k(h)(z)$ on the jet bundle $\mathcal{J}_k(E)$
 need not be positive definite for $z\in \Omega$. Thus
$\mathcal{J}_k(E)$ has no natural Hermitian metric, just a
Hermitian form.

For $\mathcal{H}$ a complex Hilbert space and $n$ a positive
integer, let $\mathcal{G}r(n,\mathcal{H})$ denote the Grassmann
manifold, the set of all $n$-dimensional subspaces of
$\mathcal{H}$.

\begin{defn}For $\Omega$ an open connected subset of $\mathbb{C}$, we
say that a map $f:\Omega\rightarrow \mathcal{G}r(n,\mathcal{H})$
is holomorphic at $\lambda_0\in \Omega$ if there exists a
neighborhood $U$ of $\lambda_0$ and $n$ holomorphic $\mathcal{H}$-
valued functions $\sigma_1,\ldots,\sigma_n$ on $U$ such that
$f(\lambda)=\bigvee
\{\sigma_1(\lambda),\ldots,\sigma_n(\lambda)\}$ for $\lambda$ in
$U$. If this holds for each $\lambda_0\in\Omega$ then we say that
 $f$ is holomorphic on $\Omega$.\end{defn} If $f:\Omega\rightarrow
\mathcal{G}r(n,\mathcal{H})$ is a holomorphic map, then a natural
$n$-dimensional Hermitian holomorphic vector bundle $E_{f}$ is
induced over $\Omega$, namely,
$$E_{f}=\{(x,\lambda)\in \mathcal{H}\times \Omega : x\in
f(\lambda)\}$$  \mbox{and}
 $$\pi : E_{f}\rightarrow \Omega \;\;\mbox{where} \;\;\pi(x,\lambda)=\lambda.$$

\begin{defn} \label{jet:inf non:def}Let $f:\Omega\rightarrow
\mathcal{G}r(n,\mathcal{H})$ be a holomorphic map. We say that $f$
is $k$- nondegenerate if, for each $w_0\in\Omega$, there exists a
neighborhood $U$ of $w_0$ and $n$ holomorphic $\mathcal{H}$-
valued functions $\sigma_1,\ldots,\sigma_n$ on $U$ such that
$\sigma_1(w),\ldots,\sigma_n(w),\ldots,\sigma^{(k)}_1(w),\ldots
\sigma^{(k)}_n(w)$ are independent for each $w$ in the open set
$U$.  If this holds for all $k=0,1,\ldots,$ then we say that $f$
is
 nondegenerate.
\end{defn}
If $f$ is $k$ nondegenerate, then $f$ induces a holomorphic map
$$j_k(f):\Omega \rightarrow \mathcal{G}r((k+1)n,\mathcal{H})$$ such
that $j_k(f)(w)$ is  the span of
$\sigma_1(w),\ldots,\sigma^{(k)}_n(w)$. If $\sigma$ is a frame for
$E_{f}$ on $U$, let
$j_k(\sigma)=\{\sigma_1,\ldots,\sigma_n,\ldots,\sigma^{(k)}_1\,\ldots,\sigma^{(k)}_n\}$
be the induced frame for $E_{j_k(f)}$. Then $\mathcal{J}_k(E_{f})$
and $E_{j_k(f)}$ are naturally equivalent Hermitian holomorphic
bundles by identifying $\sigma_{ir}$ with $\sigma_i^{(r)}$, since
$\langle\sigma_{ir},\sigma_{js}\rangle={\partial^{r+s}\langle\sigma_i,\sigma_j}\rangle/{\partial
z^r \partial \bar{z}^s}=
\langle\sigma_i^{(r)},\sigma_j^{(s)}\rangle$. In this case
$\mathcal{J}_k(h)$ is a Hermitian metric for
$\mathcal{J}_k(E_{f})$,that is, $\mathcal{J}_k(h)$ is positive
definite.

\begin{defn}
Let $\mathcal{H}$ be a Hilbert space and $\Omega$ be a bounded
domain in $\mathbb{C}^m$. Let ${\mathfrak{G}}_n(\Omega,\mathcal
H)$ be the set of all Hermitian holomorphic vector bundles of rank
$n$ over $\Omega$ which arise as a pull-backs of the tautological
bundle by nondegenerate holomorphic maps. That is, for any
nondegenerate holomorphic map
$f:\Omega\to\mathcal{G}r(n,\mathcal{H})$ the vector bundle
$E_f=\{(x,\lambda)\in \mathcal{H}\times \Omega : x\in
f(\lambda)\}$ is in ${\mathfrak{G}}_n(\Omega,\mathcal H)$.
\end{defn}
\begin{rem}
If $E_f$ is in ${\mathfrak{G}}_n(\Omega,\mathcal H)$, then the
preceding calculation shows that $\mathcal{J}_k(E_f)$ is in
${\mathfrak{G}}_{n(k+1)}(\Omega,\mathcal H)$.
\end{rem}

\section{Line Bundles}
 Let $\mathcal{L}_{f}$ be a Hermitian holomorphic  line bundle
 over a bounded domain
$\Omega\subset\mathbb{C}$. Assume that $\mathcal{L}_f\in
{\mathfrak{G}}_1(\Omega,\mathcal H)$. Let
$\mathcal{J}_{k}(\mathcal{L}_{f})$ be a jet bundle of rank $k+1$
obtained from $\mathcal{L}_{f}$. Let $\sigma$ be a frame for
$\mathcal{L}_{f}$ over an open subset $\Omega_0$ of $\Omega$. A
frame for $\mathcal{J}_{k}(\mathcal{L}_{f})$ over the open set
$\Omega_0$ is easily seen to be the set $\{\sigma,\frac{\partial
\sigma}{\partial z},\frac {\partial^2 \sigma}{\partial
z^2},\ldots,\frac{\partial^{k}\sigma}{\partial z^k}\}$. Let $h$ be
a metric for $\mathcal{L}_{f}$, which is of the form
$$h(z)=\langle \sigma(z),\sigma(z)\rangle.$$ The metric for
the jet bundle $\mathcal{J}_{k}(h)$ is then of the form
$$\mathcal{J}_{k}(h)(z)=
\begin{pmatrix}
   h(z)&  \cdots & \frac {\partial^{k}}{\partial z^k}h(z) \\
  \vdots  & \ddots & \vdots \\
  \frac{\partial^k}{\partial \overline{z}^{k}}h(z) &  \cdots & \frac{\partial^{2k}}{\partial\overline{z}^{k}\partial z^{k}}h(z) \\
\end{pmatrix}.$$ Let $\mathcal{K}_{\mathcal{J}_{k}(\mathcal{L}_{f})}$ be the
curvature of the jet bundle $\mathcal{J}_{k}(\mathcal{L}_{f})$. An
explicit formula for the curvature of a Hermitian holomorphic
vector bundle $E$ is given in \cite[proposition 2.2, pp.
79]{wells}. The curvature
$\mathcal{K}_{\mathcal{J}_{k}(\mathcal{L}_{f})}$ of the jet bundle
therefore takes the form
$$\mathcal{K}_{\mathcal{J}_{k}(\mathcal{L}_{f})}(z)=
\overline{\partial}\{(\mathcal{J}_{k}(h)(z))^{-1}\partial
\mathcal{J}_{k}(h)(z)\},$$ with respect to the metric
$\mathcal{J}_k(h)$ obtained from frame $\{\sigma,\frac{\partial
\sigma}{\partial z},\frac {\partial^2 \sigma}{\partial
z^2},\ldots,\frac{\partial^{k}\sigma}{\partial z^k}\}$. Set
$\mathds{J}^k(z)=(\mathcal{J}_{k}(h)(z))^{-1}\frac{\partial}{\partial
z} \mathcal{J}_{k}(h)(z)$ and note that

\begin{eqnarray*}(\mathcal{J}_{k}(h)(z))^{-1}\partial
(\mathcal{J}_{k}(h)(z))
 = \left(\begin{smallmatrix}
  0 & 0  & \cdots & 0 & (\mathds{J}^k(z))_{1,k+1} \\
  1 & 0  & \cdots & 0 & (\mathds{J}^k(z))_{2,k+1} \\
  \vdots &  \vdots & \ddots & \vdots & \vdots \\
  0 & 0  & \cdots & 1 & (\mathds{J}^k(z))_{k+1,k+1} \\
\end{smallmatrix} \right)dz,\end{eqnarray*}
where $(\mathds{J}^k(z))_{i,k+1}$ is the $(i,k+1)^{\rm th}$ entry
of the matrix $\mathds{J}^k(z)$. The matrix product in the first
equation is of the form $A^{-1}B$, where the first $k$ columns of
$B$ are the last $k$ column of $A$.

Therefore the curvature of the jet bundle
$\mathcal{J}_k(\mathcal{L}_f)$
 is seen to be of the form
$$\mathcal{K}_{\mathcal{J}_k (\mathcal{L}_{f})}(z)=
\begin{pmatrix}
  0 &   \cdots & 0 & b_1(z) \\
  \vdots  & \ddots &\vdots & \vdots \\
  0 &  \cdots & 0 & b_{k}(z) \\
  0 &  \cdots & 0 & \mathcal{K}_{\det(\mathcal{J}_k \mathcal{L})}(z) \\
\end{pmatrix}d\overline{z}\wedge  dz,
$$
 where $b_i(z)= \frac{\partial}{\partial\bar{z}}[(\mathds{J}^k(z))_{i,k+1}],\;\;1\leq i\leq
 k$.

\begin{thm}\label{theo eql} As before, let $\mathcal{L}_{f}$ and $\mathcal{L}_{\tilde{f}}$ be
two Hermitian holomorphic line bundles  over a bounded domain
$\Omega\subset\mathbb{C}$. Let $\mathcal{J}_k(\mathcal{L}_{f})$
and $\mathcal{J}_k(\mathcal{L}_{\tilde f})$ be the corresponding
jet bundles of rank $k+1$. If $\mathcal{J}_k(\mathcal{L}_{f})$ is
locally equivalent to $\mathcal{J}_k(\mathcal{L}_{\tilde f})$,
then $\mathcal{J}_{k-1}(\mathcal{L}_{f})$ is locally equivalent to
$\mathcal{J}_{k-1}(\mathcal{L}_{\tilde f})$.
\end{thm}

\begin{proof} Since $\mathcal{J}_k(\mathcal{L}_{f})$ and
$\mathcal{J}_k(\mathcal{L}_{\tilde f})$ are locally equivalent,
for each $z_0\in \Omega$, there exists a neighborhood $\Omega_0$
and a holomorphic bundle map $\phi\colon
\mathcal{J}_k(\mathcal{L}_{f})_{|\Omega_0}\rightarrow
\mathcal{J}_k(\mathcal{L}_{\tilde f})_{|\Omega_0}$ such that
$\phi$ is an isomorphism. Let
$\mathcal{J}_k(\sigma)=\{\sigma,\frac{\partial \sigma}{\partial
z},\frac {\partial^2 \sigma}{\partial
z^2},\ldots,\frac{\partial^{k}\sigma}{\partial z^k}\}$ and
$\mathcal{J}_k(\tilde{\sigma})=\{\tilde \sigma,\frac{\partial
\tilde \sigma}{\partial z},\frac {\partial^2 \tilde
\sigma}{\partial z^2},\ldots,\frac{\partial^{k}\tilde
\sigma}{\partial z^k}\}$ be frames for
$\mathcal{J}_k(\mathcal{L}_{f})$ and
$\mathcal{J}_k(\mathcal{L}_{\tilde f})$ over the open subset
$\Omega_0$ of $\Omega$ respectively.

Now \begin{eqnarray}\label{matrix of phi 1}
\phi(\tfrac{\partial^{j}\sigma}{\partial z^j}(z))=
\sum^{k}_{i=0}{\phi_{ij}(z)\tfrac{\partial^{i}\tilde
\sigma}{\partial z^i}(z)}.
\end{eqnarray}
So the matrix representing $\phi$ with respect to the two frames
$\mathcal{J}_k(\sigma)$ and $\mathcal{J}_k(\tilde{\sigma})$ is
\begin{eqnarray}
\phi(z)= \left(\begin{smallmatrix}
  \phi_{0,0}(z)  & \cdots & \phi_{0,k}(z) \\
  \vdots & \ddots & \vdots \\
  \phi_{k,0}(z)  & \cdots & \phi_{k,k}(z) \\
\end{smallmatrix}\right).
\end{eqnarray} Therefore we can write \begin{eqnarray}\label{matrix of phi
3}\big(\phi(\sigma(z)),\phi(\tfrac{\partial \sigma}{\partial
z}(z)),\ldots,\phi(\tfrac{\partial^{k}\sigma}{\partial
z^k}(z))\big)= \big(\tilde \sigma(z),\tfrac{\partial \tilde
\sigma}{\partial z}(z),\ldots,\tfrac{\partial^{k}\tilde
\sigma}{\partial z^k}(z)\big) \phi(z) .\end{eqnarray} But we know
that

\begin{eqnarray} \label{cur 1} \phi(z)
\mathcal{K}_{\mathcal{J}_{k}(\mathcal{L}_{f})}(z)=
\mathcal{K}_{\mathcal{J}_{k}(\mathcal{L}_{\tilde f})}(z) \phi(z).
\end{eqnarray} Now

 \begin{eqnarray}\label{cur 2}
&&\big(\phi(z) \mathcal{K}_{\mathcal{J}_{k}(\mathcal{L}_{f})}(z)\big)_{ij} \nonumber\\
&=&  \begin{cases} 0 &\mbox{if}
\,\, 0\leq i,j\leq k-1,\\ \sum_{l=0}^{k-1}b_{l+1}(z).\phi_{i,l}(z)+
  \mathcal{K}_{\det({\mathcal{J}_{k}(\mathcal{L}_{f})})}(z).\phi_{i,k}(z) d\overline{z}\wedge dz & \mbox{if}\,\, 0 \leq i\leq k, j=k . \end{cases}
\end{eqnarray}

 and
 \begin{eqnarray} \label{cur 3}
 \mathcal{K}_{\mathcal{J}_{k}(\mathcal{L}_{\tilde
f})}(z) \phi(z) =\left(\begin{smallmatrix}
  b_1(z).\phi_{k,0}(z)  & \cdots  & b_1(z).\phi_{k,k}(z) \\
  \vdots  & \ddots & \vdots \\
  b_{k-1}(z).\phi_{k,0}(z)
  & \cdots  &  b_{k-1}(z).\phi_{k,k}(z) \\
  \mathcal{K}_{\det(\mathcal{J}_k (\mathcal{L}_{\tilde f}))}(z).\phi_{k,0}(z)
& \cdots  & \mathcal{K}_{\det(\mathcal{J}_k (\mathcal{L}_{\tilde f}))}(z).\phi_{k,k}(z) \\
\end{smallmatrix}\right) d\overline{z}\wedge dz \end{eqnarray}
 Hence from equations \eqref{cur 1}, \eqref{cur 2} and \eqref{cur 3},
 it follows that
$$\phi_{k,0}(z)=\phi_{k,1}(z)=\cdots=\phi_{k,k-1}(z)=0.$$
So the bundle map $\phi$ has the form
\begin{eqnarray}\label{matrix of phi 4} \phi (z)=
\left(\begin{smallmatrix}
  \phi_{0,0}(z) & \phi_{0,1}(z) & \cdots  & \phi_{0,k}(z) \\
  \vdots & \vdots & \ddots &  \vdots \\
  \phi_{k-1,0}(z) & \phi_{k-1,1}(z) & \cdots &  \phi_{k-1,k}(z) \\
  0 & 0 & \cdots &  \phi_{k,k}(z) \\
\end{smallmatrix}\right) \end{eqnarray} with respect to the frames
$\mathcal{J}_k(\sigma)$ and ${\mathcal{J}}_k(\tilde{\sigma})$.
Finally from equations \eqref{matrix of phi 3} and \eqref{matrix
of phi 4}, we see that
$$\phi_{|\mathcal{J}_{k-1}(\mathcal{L}_{f})_{|\Omega_0}}:\mathcal{J}_{k-1}(\mathcal{L}_{f})_{|\Omega
_0}\to \mathcal{J}_{k-1}(\mathcal{L}_{\tilde f})_{|\Omega_0}.$$
Since $\phi$ is a bundle isomorphism, it follows that
$$\phi_{|\mathcal{J}_{k-1}(\mathcal{L}_{f})_{|\Omega_0}}:\mathcal{J}_{k-1}(\mathcal{L}_{f})_{|\Omega_0}\to
\mathcal{J}_{k-1}(\mathcal{L}_{\tilde f})_{|\Omega_0}$$ is also a
bundle isomorphism.
\end{proof}

\begin{cor} Let $\mathcal{L}_{f}$ and $\mathcal{L}_{\tilde f}$ be
Hermitian holomorphic  line bundles. Let
$\mathcal{J}_k(\mathcal{L}_{f})$ and
$\mathcal{J}_k(\mathcal{L}_{\tilde f})$ be the corresponding jet
bundles of rank $k+1$. The two jet bundles
$\mathcal{J}_k(\mathcal{L}_{f})$ and
$\mathcal{J}_k(\mathcal{L}_{\tilde f})$ are locally equivalent
 as Hermitian holomorphic
vector bundles if and only if the two line bundles
$\mathcal{L}_{f}$ and $\mathcal{L}_{\tilde f}$ are locally
equivalent as Hermitian holomorphic vector bundles.
\end{cor}
\begin{proof} Suppose $\mathcal{J}_k(\mathcal{L}_{f})$ and
$\mathcal{J}_k(\mathcal{L}_{\tilde f})$ are locally equivalent.
Then for each $z_0\in \Omega$ there exists a neighborhood
$\Omega_0$ and a holomorphic map $\phi\colon
\mathcal{J}_k(\mathcal{L}_{f})_{|\Omega_0}\rightarrow
\mathcal{J}_k(\mathcal{L}_{\tilde f})_{|\Omega_0}$ such that
$\phi$ is an isomorphism.

Using Theorem \ref{theo eql},
$\phi_{|\mathcal{J}_{k-1}(\mathcal{L}_{f})_{|\Omega_0}}:\mathcal{J}_{k-1}(\mathcal{L}_{f})_{|\Omega
_0}\to \mathcal{J}_{k-1}(\mathcal{L}_{\tilde f})_{|\Omega_0}$ is
an isomorphism.
\smallskip
Since
$\phi_{|\mathcal{J}_{k-1}(\mathcal{L}_{f})_{|\Omega_0}}:\mathcal{J}_{k-1}(\mathcal{L}_{f})_{|\Omega
_0}\to \mathcal{J}_{k-1}(\mathcal{L}_{\tilde f})_{|\Omega_0}$ is
an isomorphism, by the same argument which is given in the proof
of the Theorem  \ref{theo eql}, it follows that
$$\phi_{|\mathcal{J}_{k-2}(\mathcal{L}_{f})_{|\Omega_0}}:\mathcal{J}_{k-2}(\mathcal{L}_{f})_{|\Omega
_0}\to \mathcal{J}_{k-2}(\mathcal{L}_{\tilde f})_{|\Omega_0}$$ is
an isomorphism. Repeating this argument, we see that $\phi$ is an
isomorphism from ${\mathcal{L}_{f}}_{|\Omega_0}$ to
${\mathcal{L}_{\tilde f}}_{|\Omega_0}$.
\end{proof}

 Let $A$ be an $n\times n$ matrix and $A_{\hat i,\hat j}$ be the
 $(n-1)\times (n-1)$ matrix which is obtained from $A$ by
 removing the $i^{\rm th}$ row and $j^{\rm th}$ column of the matrix
 $A$.

  \begin{lem} \label{lin lemma 1}Let $A$ be an $n\times n$ matrix and $B$
  be the
 $(n-2)\times(n-2)$ matrix which is obtained from $
 A$ by removing the last two rows and last two
 columns of $A$.
 Then
 $$\det(A_{\hat n,\hat n}) \det(A_{\widehat {n-1},\widehat {n-1}})-
 \det(A_{\hat n,\widehat {n-1}}) \det(A_{\widehat {n-1},\hat n})=
 \det(B) \det(A).$$
 \end{lem}

 \begin{proof} {\sf Case(1):} suppose $B$ is invertible.
 Let $$A=
\left(\begin{smallmatrix}
  a_{1,1}  & \cdots & a_{1,n} \\
  \vdots  & \ddots &  \vdots \\
  a_{n,1} & \cdots &  a_{n,n} \\
\end{smallmatrix}\right) $$ and
$$
x_1=(a_{1,n-1}, a_{2,n-1},\ldots,a_{n-2,n-1})^{\rm tr} ,
x_2=(a_{1,n},a_{2,n},\ldots,a_{n-2,n})^{\rm tr}
$$
$$y_{1}=(a_{n-1,1}, a_{n-1,2}, \ldots, a_{n-1,n-2}) , y_2=(
  a_{n,1}, a_{n,2}, \ldots, a_{n,n-2}).$$
  Thus the matrix $A$ can be written in the form
$$A=
\begin{pmatrix}
  B & x_1 & x_2 \\
  y_1 & a_{n-1,n-1} & a_{n-1,n} \\
  y_2 & a_{n,n-1} & a_{n,n} \\
\end{pmatrix} .$$ In this notation, we have the following
equalities:

\begin{eqnarray}\label{det 1} \det(A_{\hat n,\hat n})&=& \det
\begin{pmatrix}
  B & x_1 \\
  y_1 & a_{n-1,n-1} \nonumber\\
\end{pmatrix}\\&&\nonumber\\
&=& \det(B)(a_{n-1,n-1}-y_1 B^{-1}x_1),
\end{eqnarray}
\begin{eqnarray}\label{det 2} \det(A_{\widehat{n-1},\widehat{n-1}})&=& \det
\begin{pmatrix}
  B & x_2 \\
  y_2 & a_{n,n} \nonumber\\
\end{pmatrix}\\&&\nonumber\\
&=& \det(B)(a_{n,n}-y_2 B^{-1}x_2), \end{eqnarray}
\begin{eqnarray}\label{det 3} \det(A_{\hat n,\widehat {n-1}})&=& \det
\begin{pmatrix}
  B & x_2 \\
  y_1 & a_{n-1,n} \nonumber\\
\end{pmatrix}\\&&\nonumber\\
&=& \det(B)(a_{n-1,n}-y_1 B^{-1}x_2), \end{eqnarray}
\begin{eqnarray}\label{det 4} \det(A_{\widehat{n-1},\hat n})&=& \det
\begin{pmatrix}
  B & x_1 \\
  y_2 & a_{n,n-1} \nonumber\\
\end{pmatrix}
\\&&\nonumber\\
&=& \det(B)(a_{n,n-1}-y_2 B^{-1}x_1), \end{eqnarray} and
\begin{eqnarray} \label{det 5} \lefteqn{\det(A)}\nonumber\\&=& \det(B)
\det\left\{
\begin{pmatrix}
  a_{n-1,n-1} & a_{n-1,n} \\
  a_{n,n-1} & a_{n,n} \\
\end{pmatrix}-
\begin{pmatrix}
  y_1 \\
  y_2 \\
\end{pmatrix}
\scriptstyle{B}^{-1}
\begin{pmatrix}
  x_1 & x_2 \\
\end{pmatrix}
\right\}\nonumber \\&&\nonumber\\
 &=&\det(B)\det
\begin{pmatrix}
  a_{n-1,n-1}-y_1 B^{-1}x_1 & a_{n-1,n}-y_1 B^{-1} x_2 \\
  a_{n,n-1}-y_2 B^{-1} x_1& a_{n,n}-y_2 B^{-1} x_2 \\
\end{pmatrix}%
 \nonumber\\&&\nonumber\\
&=&\det(B)\left\{(a_{n-1,n-1}-y_1
\scriptstyle{{B}^{-1}}x_1)(a_{n,n}-y_2 \scriptstyle{{B}^{-1}}
x_2)-(a_{n-1,n}-y_1 \scriptstyle{{B}^{-1}} x_2) (a_{n,n-1}-y_2
\scriptstyle{{B}^{-1}} x_1) \right\}.\end{eqnarray} From equation
\eqref{det 1},\eqref{det 2},\eqref{det 3},\eqref{det 4} and
\eqref{det 5}, it follows that

$$\det(A)=\det(B) \left\{\frac{\det(A_{\hat n,\hat n})\det(A_{\widehat{n-1},\widehat{n-1}})}{(\det B)^2}-
\frac{\det(A_{\widehat{n-1},\hat n}) \det(A_{\hat n,\widehat
{n-1}})}{(\det B)^2}\right\},$$ that is,
\begin{eqnarray}\label{det 6}
\det(A_{\hat n,\hat n})\det(A_{\widehat {n-1},\widehat{n-1}})-
\det(A_{\widehat{n-1},\hat n}) \det(A_{\hat n,\widehat {n-1}})=
\det(B) \det(A).
 \end{eqnarray}

\noindent{\sf Case(2):} Suppose $B$ is not invertible. Then there
exists a sequence of invertible matrices $B_m$ that approximate
$B$, that is,
  $\|B_m-B\|\to
0$, as $m\to \infty$. Let

$$A_m=
\begin{pmatrix}
  B_m & x_1 & x_2 \\
  y_1 & a_{n-1,n-1} & a_{n-1,n} \\
  y_2 & a_{n,n-1} & a_{n,n} \\
\end{pmatrix}$$
clearly $\|A_m-A\|\to 0$ as $m\to \infty$. From the proof of the
previous case, we have
$$\det\{(A_m)_{\hat n,\hat n}\}
\det\{(A_m)_{\widehat{n-1},\widehat{n-1}}\}-\det\{(A_m)_{\hat
n,\widehat{n-1}}\} \det\{(A_m)_{\widehat{n-1},\hat n}\}= \det(B_m)
\det(A_m) .$$
Since determinant is a continuous function, taking
$m \to \infty$, it follows that
$$\det(A_{\hat n,\hat n})\det(A_{\widehat {n-1},\widehat{n-1}})-
\det(A_{\widehat{n-1},\hat n}) \det(A_{\hat n,\widehat {n-1}})=
\det(B) \det(A).$$
\end{proof}

\begin{prop}The curvature of the determinant bundle
$\det\,\mathcal{J}_k(\mathcal{L}_{f})$ is given by the following
formula
 $$\mathcal{K}_{\det\mathcal{J}_k (\mathcal{L}_{f})}(z)=
 \frac{(\det\mathcal{J}_{k-1} h)(z) (\det\mathcal{J}_{k+1} h)(z)}
 {(\det \mathcal{J}_k h)^2(z)}\;\;d\overline z\wedge dz.$$
\end{prop}
\begin{proof}
The curvature of the determinant bundle
$\det(\mathcal{J}_k(\mathcal{L}_{f}))$ is
 $$\mathcal{K}_{\det\mathcal{J}_k (\mathcal{L}_{f})}(z)=
 \frac{(\det\mathcal{J}_k h)(z) (\frac{\partial^2}{\partial z\partial \overline z}\det\mathcal{J}_k h)(z)-
 (\frac{\partial}{\partial \overline z} \det\mathcal{J}_k h)(z)
 (\frac{\partial}{\partial z}
 \det\mathcal{J}_k h)(z)}{(\det \mathcal{J}_k h)^2(z)}\;\; d\overline z\wedge
 dz.$$
Here$$\mathcal{J}_k h=
\big(\!\big(\tfrac{\partial^{i+j}}{\partial\overline {z}^i
\partial z^j}h\big)\!\big)_{i,j=0}^{k}\;\;
\mbox{and}\;\; \mathcal{J}_{k+1} h=
\big(\!\big(\tfrac{\partial^{i+j}}{\partial\overline {z}^i
\partial z^j}h\big)\!\big)_{i,j=0}^{k+1}.
$$
Now, we have \begin{eqnarray} \label{matrix 1}
\frac{\partial}{\partial z}(\det\mathcal{J}_k h)
 = \det((\mathcal{J}_{k+1}h)_{\widehat{k+2},\widehat{k+1}}),
\end{eqnarray}
\begin{eqnarray} \label{matrix 2}
\frac{\partial}{\partial \overline z}(\det\mathcal{J}_k h)
 = \det((\mathcal{J}_{k+1}h)_{\widehat{k+1},\widehat{k+2}}),
\end{eqnarray}
and
\begin{eqnarray} \label{matrix 3}
\frac{\partial^2}{\partial\overline {z}\partial
z}(\det\mathcal{J}_k h) =
\det((\mathcal{J}_{k+1}h)_{\widehat{k+1},\widehat{k+1}}),
\end{eqnarray}
Finally, note that
\begin{eqnarray}\label{matrix 4} \det\mathcal{J}_k h= \det
((\mathcal{J}_{k+1}h)_{\widehat{k+2},\widehat{k+2}}).
\end{eqnarray} By Lemma \ref{lin lemma 1}, we obtain
\begin{eqnarray} \label{matrix 5}
\det (\mathcal{J}_{k-1}h) \det (\mathcal{J}_{k+1}h) &=&
{\det}({\!(\mathcal{J}_{k+1}h)}_{\widehat{k+2},\widehat{k+2}})
\det (\!(\mathcal{J}_{k+1}h)_{\widehat{k+1},\widehat{k+1}})
\nonumber\\&&~~~~~~~ - \det
(\!(\mathcal{J}_{k+1}h)_{\widehat{k+2},\widehat{k+1}}) \det
(\!(\mathcal{J}_{k+1}h)_{\widehat{k+1},\widehat{k+2}}).
\end{eqnarray} From equations \eqref{matrix 1}, \eqref{matrix
2}, \eqref{matrix 3}, \eqref{matrix 4} and \eqref{matrix 5}, it
follows that
\begin{eqnarray*}\lefteqn{(\det\mathcal{J}_{k-1} h)(z)
(\det\mathcal{J}_{k+1}h)(z)}~~~~~~~~~~~~~~~~~~~~~\\&=&
(\det\mathcal{J}_k h)(z) (\tfrac{\partial^2}{\partial z\partial
\bar z}\det\mathcal{J}_k h)(z)
 (\tfrac{\partial}{\partial \bar z} \det\mathcal{J}_k h)(z)
 (\tfrac{\partial}{\partial z}
 \det\mathcal{J}_k h)(z).\end{eqnarray*}
Hence
$$\mathcal{K}_{\det\mathcal{J}_k (\mathcal{L}_{f})}(z)=
 \frac{(\det\mathcal{J}_{k-1} h)(z) (\det\mathcal{J}_{k+1} h)(z)}
 {(\det \mathcal{J}_k h)^2(z)}\;\;d\overline z\wedge dz.$$
\end{proof}
\begin{cor} Let $\mathcal{L}_{f}$ and $\mathcal{L}_{\tilde f}$ be
  Hermitian holomorphic  line bundles over a domain
 $\Omega \subset \mathbb{C}$. The following statements are
 equivalent:
 \begin{enumerate}
\item[(1)]
 $\det\mathcal{J}_k (\mathcal{L}_{f})$
 is locally equivalent to $\det\mathcal{J}_k (\mathcal{L}_{\tilde f})$ and
$\det\mathcal{J}_{k+1}( \mathcal{L}_{f})$ is locally equivalent to
$\det\mathcal{J}_{k+1} (\mathcal{L}_{\tilde f})$, for some $k\in
\mathbb{N}$ \item[(2)] $\mathcal{L}_{f}$ is locally equivalent to
$\mathcal{L}_{\tilde f}$. \end{enumerate}
 \end{cor}
\section{Rank $n$-Vector Bundles}
We first recall some well known facts from linear algebra.

\begin{lem}
Let $A, B, C$ and $D$ be matrices of size $n\times n, n\times m,
m\times n$ and $m\times m$ respectively.

\begin{enumerate}
\item[(i)]\label{l23}\cite[pp. 138]{rao} If $A, D$ and $D- C
A^{-1} B$ are invertible, then $
\begin{pmatrix}
  A & B \\
  C & D \\
\end{pmatrix}$ is invertible and $$
\begin{pmatrix}
  A & B \\
  C & D \\
\end{pmatrix}^{-1}=
\begin{pmatrix}
  (A-BD^{-1}C)^{-1} & -A^{-1}B(D-CA^{-1}B)^{-1} \\
  -D^{-1}C(A-BD^{-1}C)^{-1} & (D-CA^{-1}B)^{-1} \\
\end{pmatrix}.
$$
\end{enumerate}
\item[(ii)]\label{l24}\cite[pp. 246]{rao}
If $A$ is invertible then
$$\det
\begin{pmatrix}
  A & B \\
  C & D \\
\end{pmatrix}=
\det(A) \det(D-CA^{-1}B)$$

\item[(iii)] \label{l25}\cite[pp. 247]{rao} If $D$ is invertible then
$$\det
\begin{pmatrix}
  A & B \\
  C & D \\
\end{pmatrix}=
\det(D) \det(A-BD^{-1}C).$$

\end{lem}

\begin{lem}\label{quo 1}\cite[pp. 240]{cd}
If $V$ is a proper, non-zero subspace of an inner product space
$W$ then it induces an inner product on the quotient $W/V$ by
\begin{eqnarray*}([w_1],[w_2])&=& ||v_1\wedge\ldots\wedge
v_n||^{-2}(v_1\wedge\ldots\wedge v_n\wedge
w_1,v_1\wedge\ldots\wedge v_n\wedge w_2)\end{eqnarray*}  where
$[w_1], [w_2]$ denote the equivalence classes of $w_1$ and $w_2$
respectively in $W/V$ and $\{v_1,\ldots,v_n\}$ is a basis for $V$.
\end{lem}

\begin{lem}\label{quo 2}
Let $W$ be an inner product space and let $V$ be a subspace of
$W$. Let $\{e_1,\ldots,e_r\}$ be a basis of $V$ and
$\{e_1,\ldots,e_r,e_{r+1},\ldots,e_n\}$ be a  basis of $W$
extending the basis of $W$. Suppose $$\sigma_i=
e_1\wedge\ldots\wedge e_r\wedge e_i,\;\; r+1\leq i\leq n$$ and

$\begin{array}{cc}
  A=\big(\!\big(\langle e_i,e_j\rangle\big)\!\big)_{1\leq i,j\leq r},&
  B=\big(\!\big(\langle e_i,e_j\rangle\big)\!\big)_{r+1\leq i\leq n,1\leq j\leq r},\\
 C=\big(\!\big(\langle e_i,e_j\rangle\big)\!\big)_{1\leq i\leq r,\, r+1\leq j\leq n}, &
  D=\big(\!\big(\langle e_i,e_j\rangle\big)\!\big)_{r+1\leq i,j\leq n},\\
\end{array}$ \\ $$\textbf{A}_{\sigma}=\big(\!\big(\langle
\sigma_i,\sigma_j\rangle\big)\!\big)_{r+1\leq i,j\leq n}.$$ Then
$$\det\big(\!\big(\langle
e_i,e_j\rangle\big)\!\big)_{1\leq i,j\leq n}=\det
\begin{pmatrix}
  A & B \\
  C & D \\
\end{pmatrix}
= \frac{\det (\textbf{A}_{\sigma})}{(\det A)^{n-r-1}}.$$

\end{lem}
\begin{proof}
Suppose $x_i=(\langle e_1,e_i\rangle,\ldots,\langle
e_r,e_i\rangle)$ and $y_i= \bar{x}_i^{\rm tr},$ $r+1\leq i\leq n$.
\begin{eqnarray*}
\langle \sigma_i,\sigma_j\rangle &=& \det
\begin{pmatrix}
  A & y_i \\
  x_j & \langle e_i,e_j\rangle \\
\end{pmatrix}\\
 &=& \det (A)(\langle e_i,e_j\rangle- x_j A^{-1}y_i).
\end{eqnarray*}
Next, note that
\begin{eqnarray*}
\det\big(\!\big(\langle e_i,e_j\rangle\big)\!\big)_{1\leq i,j\leq
n}&=&\det
\begin{pmatrix}
  A & B \\
  C & D \\
\end{pmatrix}
\\&=& \det(A) \det (D-C A^{-1}B)\\
&=&\det(A)\det\big(\!\big(\langle e_i,e_j\rangle- x_j
A^{-1}y_i\big)\!\big)_{r+1\leq i,j\leq n}\\
&=& \det(A)\det \big(\!\big({\langle
\sigma_i,\sigma_j\rangle}/{\det(A)}\big)\!\big)_{r+1\leq i,j\leq
n}\\
&=&\frac{\det (\textbf{A}_{\sigma})}{(\det A)^{n-r-1}}.
\end{eqnarray*}
\end{proof}
\begin{prop} Let $E$ be a Hermitian holomorphic vector
bundle of rank $n$ over a bounded domain $\Omega$ in
$\mathbb{C}^m$ and let $F$ be a subbundle of $E$ of rank $r$. Then
$$h_{\det (E/F)}= \frac{h_{\det E}}{h_{\det F}}$$ where $h_{\det
E}$, $h_{\det (E/F)}$ and $h_{\det F}$ are the metrics of $\det
E$, $\det F$ and $\det E/F$ respectively.
\end{prop}
\begin{proof}
Let $\{s_1,\ldots,s_r\}$ be a frame for $F$ over an open subset
$U$ of $\Omega$ and let $\{s_1,\ldots,s_r,s_{r+1},\\ \ldots,s_n\}$
be a frame of $E$ obtained by extending the frame of $F$. The
quotient  $E/F$ admits a frame of the form
$\{[s_{r+1}],\ldots,[s_n]\}$, where $[s_i],r+1\leq i\leq n,$
denotes the equivalence class of $s_i$ in $E/F$. Let
$h_{E}=\big(\!\big(\langle
s_j,s_i\rangle\big)\!\big)_{i,j=1}^{n}$,
$h_{F}=\big(\!\big(\langle s_j,s_i\rangle\big)\!\big)_{i,j=1}^{r}$
and $h_{E/F}=\big(\!\big(\langle
[s_j],[s_i]\rangle\big)\!\big)_{i,j=r+1}^{n}$ be the metrics of
$E$, $F$ and $E/F$ respectively. Then by the definition of the
determinant bundle $h_{\det E}=\det h_E$, $h_{\det F}=\det h_F$
and $h_{\det E/F}=\det h_{E/F}$. By Lemma \ref{quo 1} and Lemma
\ref{quo 2}, we have
\begin{eqnarray*}
h_{\det E/F}&=&\det h_{E/F}\\
&=&\det\big(\!\big(\langle
[s_j],[s_i]\rangle\big)\!\big)_{i,j=r+1}^{n}\\
&=&\det \left(\!\!\!\left(\frac{\langle s_1\wedge\ldots\wedge
s_r\wedge s_{j},s_1\wedge\ldots\wedge s_r\wedge s_{i}\rangle}
{||s_1\wedge,\ldots\wedge
s_r||^2}\right)\!\!\!\right)_{i,j=r+1}^{n}\\
&=&\frac{\det \big(\!\big(\langle s_1\wedge\ldots\wedge s_r\wedge
s_{j},s_1\wedge\ldots\wedge s_r\wedge
s_{i}\rangle\big)\!\big)_{i,j=r+1}^{n}}{(\det h_F)^{n-r}}\\
&=&\frac{h_{\det E}}{h_{\det F}}.
\end{eqnarray*}
\end{proof}

\begin{cor}
Let $0\to F\to E \to E/F\to 0$ be an exact sequence of Hermitian
holomorphic vector bundles. Then
$$\mathcal{K}_{\det(E/F)}=\mathcal{K}_{\det(E)}-\mathcal{K}_{\det(F)}$$
which is equivalent to
$$\rm{trace}(\mathcal{K}_{E/F})=\rm{trace}(\mathcal{K}_{E})-\rm{trace}(\mathcal{K}_{F}).$$
\end{cor}\vspace{3mm}

 Let $E_{f}$ be a Hermitian holomorphic vector bundle of rank $n$ over an open subset
$\Omega$ in $\mathbb{C}$ and let
$E_f\in{\mathfrak{G}}_n(\Omega,\mathcal H)$. Let
$\{\sigma_1,\ldots,\sigma_n\}$ be a frame for $E_f$ over an open
subset $\Omega_0$ of $\Omega$. Let $h$ be a metric for $E_f$ which
is defined as $$h(z)= \big(\!\big(\langle
\sigma_j(z),\sigma_i(z)\rangle\big)\!\big)_{i,j=1}^{n}
$$
 We define $F_i^k$ for each $1\leq k < \infty$ and $1\leq
i\leq n$ by $$F_i^k=\sigma_1\wedge\ldots\wedge
\sigma_n\wedge\ldots\wedge\frac{\partial^{k-1}\sigma_n}{\partial
z^{k-1}}\wedge \frac{\partial^k\sigma_i}{\partial z^k}, $$ where
wedge products between $\sigma_i's$ and their derivatives are
taken in the Hilbert space $\wedge \mathcal{H}$. Let $h_k$ be the
matrix
$$
h_{k}(z)=\big (\!\big (\langle F_j^k(z),F_i^k(z)\rangle\big
)\!\big )_{i,j=1}^{n}
$$

\begin{prop}\label{l26}
Let $E_f$ be a Hermitian holomorphic vector bundle of rank $n$
over $\Omega\subset\mathbb{C}$. Then the curvature
$\mathcal{K}_{E_f}$ of $E_f$ is given by
$$\mathcal{K}_{E_f}(z)=(\det h(z))^{-1}h(z)^{-1}h_1(z)\;\; d\bar{z}\wedge dz.$$
\end{prop}
\begin{proof}
 Set $x_i=\left(\frac{\partial}{\partial
\bar{z}}\langle \sigma_1,\sigma_i\rangle, \ldots
,\frac{\partial}{\partial \bar{z}}\langle
\sigma_n,\sigma_i\rangle\right)$ and $y_i= \bar{x}_i^{\rm tr},$
$1\leq i\leq n.$ For $1\leq i,j\leq n$
\begin{eqnarray*}
\langle F_j^1(z),F_i^1(z)\rangle&=&\det
\begin{pmatrix}
  h(z) & y_j\\
  x_i& \frac{\partial^2}{\partial z \partial
\bar{z}}\langle \sigma_j(z),\sigma_i(z)\rangle  \\
\end{pmatrix}\\
&=& \det(h(z))\left(\tfrac{\partial^2}{\partial z \partial
\bar{z}}\langle \sigma_j(z),\sigma_i(z)\rangle- x_i h(z)^{-1}
y_j\right).\\
\end{eqnarray*}

Now we can derive the formula for the curvature of the vector
bundle $E_f$:
\begin{eqnarray*}
\mathcal{K}_{E_f}(z)&=& h^{-1}(z)\left\{\bar{\partial}\partial
h(z)-\bar{\partial}h(z)h^{-1}(z)\partial
h(z)\right\}\\
&=& h^{-1}(z)\big(\!\big(\tfrac{\partial^2}{\partial z\partial
\bar{z}}\langle
\sigma_j(z),\sigma_i(z)\rangle- x_i h(z)^{-1} y_j\big)\!\big)_{i,j=1}^{n} d\bar{z}\wedge dz\\
&=& h^{-1}(z)\big(\!\big( (\det
h(z))^{-1}\langle F_j^1(z),F_i^1(z)\rangle\big)\!\big)_{i,j=1}^{n}d\bar{z}\wedge dz\\
&=&(\det h(z))^{-1}h^{-1}(z)h_1(z) \;\;d\bar{z}\wedge
dz\hspace{1cm}
\end{eqnarray*}
\end{proof}

\begin{cor}
Let $E_f$ be a vector bundle of rank $n$ over a bounded domain
$\Omega\subset\mathbb{C}$. Then the curvature of the bundle $E_f$
is of rank $r$ if and only if exactly $r$ elements are independent
from the set $\{F_1^1,\ldots,F_n^1\}$ of $n$ elements.
\end{cor}
\begin{proof}
By Lemma \ref{l26} the rank of the curvature of the bundle $E$ is
same as the rank of $h_1$. But rank of $h_1$ is $r$ if and only if
$r$ elements are independent from the set $\{F_1^1,\ldots,F_n^1\}$
of $n$ elements.
\end{proof}
A result from \cite[page 238, Lemma 4.12]{cd}, which appeared to
be mysterious, now follows from the formula derived for the rank
of the curvature. Thus we have the following corollary:
\begin{cor}
Let $E_f$ be a vector bundle of rank $n$ over a bounded domain
$\Omega$ in $\mathbb{C}$. Then the rank of the  curvature
$\mathcal{K}_{\mathcal{J}_k(E_f)}$ of the jet bundle
$\mathcal{J}_k(E_f)$, $1\leq k<\infty$, is at most $n$.
\end{cor}

\subsection{Curvature Formula in General}
Let $E_{f}\stackrel{\pi}\to\Omega$ be a Hermitian holomorphic
vector bundle of rank $n$. Let $\{s_1,\cdots,s_n\}$ be a local
frame of $E_{f}$ over an open subset $\Omega_0$ of $\Omega$. Let
$h$ be a metric for $E_f$ which is defined as $$h(z)=
\big(\!\big(\langle
s_i(z),s_j(z)\rangle\big)\!\big)_{i,j=1}^{n}.$$ For $1\leq p\leq
n$ and $1\leq j\leq m$ set
$$\tau_p^j= s_1\wedge\cdots \wedge s_n\wedge \frac{\partial
s_p}{\partial z_j}.$$
For  $1\leq i,j\leq m$ set
$$h_{ij}(z)= \big(\!\big(\langle\tau_p^i(z),\tau^j_q(z)\rangle\big)\!\big)_{p,q=1}^{n}.$$

\begin{prop}
Let $E_f\stackrel{\pi}\to \Omega$ be a Hermitian holomorphic
vector bundle of rank $n$ over a domain $\Omega$ in
$\mathbb{C}^m$. Then curvature $\mathcal{K}_{E_f}$ of the vector
bundle $E_f$ is given by
$$\mathcal{K}_{E_{f}}(z)=(\det h(z))^{-1} h^{-1}(z)\sum_{i,j=1}^{m}h_{ij}(z)\,d\overline{z}_j\wedge d z_i.$$
\end{prop}

\begin{proof}
Set $x_p^j=\left(\frac{\partial}{\partial\overline {z}_j}\langle
s_1,s_p\rangle, \cdots ,\frac{\partial}{\partial\overline
{z}_j}\langle s_n,s_p\rangle\right)$ and
$y_p^i=\overline{x_p^i}^{\rm tr}\;\;\mbox{for}\;\; 1\leq p\leq
n.$\smallskip\vspace{2mm}

\noindent For $1\leq i,j\leq m$,
\begin{eqnarray*}
\tfrac{\partial^2h}{\partial \overline{z}_j\partial z_i}(z)-
\tfrac{\partial h}{\partial \overline{z}_j}(z)
h^{-1}(z)\tfrac{\partial h}{\partial
z_i}(z)&=&\big(\!\big(\tfrac{\partial^2}{\partial
\overline{z}_j\partial z_i}\langle
s_q(z),s_p(z)\rangle- x_p^j h(z)^{-1} y_q^i\big)\!\big)_{p,q=1}^{n}\\
&=&\big(\!\big( (\det
h(z))^{-1}\langle\tau_q^i(z),\tau_p^j(z)\rangle\big)\!\big)_{p,q=1}^{n}\\
&=&(\det h(z))^{-1} h_{ij}(z).
\end{eqnarray*}
Hence the curvature of the vector bundle $E_f$ takes the form:
\begin{eqnarray*}\mathcal{K}_{E_{f}}(z)&=&
h^{-1}(z)\sum_{i,j=1}^{m}\left(\tfrac{\partial^2h}{\partial
\bar{z}_j\partial z_i}(z)- \tfrac{\partial h}{\partial
\bar{z}_j}(z) h^{-1}(z)\tfrac{\partial h}{\partial
z_i}(z)\right)\,d\bar{z}_j\wedge d z_i\\
&=&(\det h(z))^{-1}
h^{-1}(z)\sum_{i,j=1}^{m}h_{ij}(z)\,d\overline{z}_j\wedge d
z_i.\end{eqnarray*}
\end{proof}

\subsection{Curvature of the Jet Bundle}

Let $\mathcal{J}_{k}(E_{f})$ be a jet bundle of rank $n(k+1)$ over
$\Omega$, where $\Omega$ is a bounded domain in $\mathbb{C}$. If
$\sigma=\{\sigma_1,\cdots,\sigma_n\}$ is a frame for $E_f$ then a
frame for $\mathcal{J}_{k}(E_{f})$ is of the form
 $$\mathcal{J}_k(\sigma)=\{\sigma_1,\cdots,\sigma_n,\tfrac{\partial}{\partial z}
\sigma_1,\cdots,\tfrac{\partial}{\partial
z}\sigma_n,\ldots,\tfrac{\partial^k}{\partial
z^k}\sigma_1,\ldots,\tfrac{\partial^k}{\partial z^k}\sigma_n\}.$$
\smallskip
By Lemma \ref{l26} the curvature
$\mathcal{K}_{\mathcal{J}_{k}(E_{f})}$ of the bundle
$\mathcal{J}_k(E_f)$ is given by
$$\mathcal{K}_{\mathcal{J}_{k}(E_{f})}(z)= \big(\det
\mathcal{J}_{k}(h)(z)\big)^{-1}(\mathcal{J}_{k}(h)(z))^{-1}
\begin{pmatrix}
  0_{n k\times n k} & 0_{n k \times n} \\
  0_{n\times n k} & h_{k+1}(z) \\
\end{pmatrix}d\bar z\wedge dz
$$
Let $A= \mathcal{J}_{k-1}(h),$
$$C=
\begin{pmatrix}
 \frac{\partial^{k}h}{\partial \bar{z}^{k}}, & \ldots, &
 \frac{\partial^{2k-1}h}{\partial z ^{k-1}\partial \bar{z}^{k}} \\
\end{pmatrix}, $$

$B = \bar{C}^{\rm tr}, \,\, D= \frac{\partial^{2k}}{\partial
z^{k}\partial\bar{z}^{k}}h,$
{
\begin{eqnarray*}{x_i}=\big({ \tfrac{\partial^{k}}{\partial
\bar{z}^{k}}\langle \sigma_1,\sigma_i\rangle, \ldots ,
\tfrac{\partial^{k}}{\partial \bar{z}^{k}}\langle
\sigma_n,\sigma_i\rangle,\ldots,\tfrac{\partial^{2k-1}}{\partial z
^{k-1}\partial \bar{z}^{k}}\langle
\sigma_n,\sigma_i\rangle}\big),\;1 \leq i \leq n,\end{eqnarray*}}
and finally $y_i= \bar{x}^{\rm tr}_i$, $1 \leq i \leq
n.$\smallskip

\noindent Now
\begin{eqnarray*}
D-CA^{-1}B&=& \tfrac{\partial^{2k}}{\partial z^{k}\partial\bar{z}^{k}}h-C A^{-1}B\\
&=&\big(\!\big(\tfrac{\partial^{2k}}{\partial
z^{k}\partial\bar{z}^{k}}\langle
\sigma_j,\sigma_i\rangle- x_i A^{-1} y_j \big)\!\big)_{i,j=1}^{n}\\
&=&\big(\!\big( (\det
\mathcal{J}_{k-1}h)^{-1}\langle F_j^{k},F_i^{k}\rangle\big)\!\big)_{i,j=1}^{n}\\
&=& (\det \mathcal{J}_{k-1}h)^{-1} h_{k}.
 \end{eqnarray*}
Consequently, \begin{eqnarray*} (\mathcal{J}_{k}h)^{-1}&=&
\begin{pmatrix}
  A & B \\
  C & D \\
\end{pmatrix}^{-1}\\&=&
\begin{pmatrix}
  (A-BD^{-1}C)^{-1} & -A^{-1}B(D-CA^{-1}B)^{-1} \\
  -D^{-1}C(A-BD^{-1}C)^{-1} & (D-CA^{-1}B)^{-1} \\
\end{pmatrix}\\
&=& \begin{pmatrix}
  (A-BD^{-1}C)^{-1} & -A^{-1}B(D-CA^{-1}B)^{-1} \\
  -D^{-1}C(A-BD^{-1}C)^{-1} & \det (\mathcal{J}_{k-1}h) h_{k}^{-1} \\
\end{pmatrix}.\\
\end{eqnarray*}
The  curvature of the jet bundle $\mathcal{J}_{k}(E_{f})$ is
\begin{eqnarray*} \lefteqn{\mathcal{K}_{\mathcal{J}_{k}(E_{f})}(z)}\\&=&
\begin{pmatrix}
  0_{n k\times n k} & -\big(\det \mathcal{J}_{k}(h)(z)\big)^{-1}A^{-1}(z)B(z)\big(D(z)-C(z)A^{-1}(z)B(z)\big)^{-1}h_{k+1}(z) \\
  0_{n\times n k} &\big(\det \mathcal{J}_k(h)(z)\big)^{-1} \det (\mathcal{J}_{k-1}h(z))h_{k}^{-1}(z)h_{k+1}(z) \\
\end{pmatrix}
\end{eqnarray*}
Here
$$\det\mathcal{J}_k h(z)= (\det \mathcal{J}_{k-1}h(z))^{1-n} \det h_k(z)$$
and
\begin{eqnarray*}\lefteqn{(\det\mathcal{J}_{k}h(z))^{-1}\det\mathcal{J}_{k-1}
h(z)}\\= &(\det h(z))^{n(1-n)^{k-1}}& \!\!\!\!\!(\det
h_1(z))^{n(1-n)^{k-2}}\cdots (\det
h_{k-2}(z))^{n(1-n)}(\det h_{k-1}(z))^n (\det
h_{k}(z))^{-1}.\end{eqnarray*}

\subsection{The Trace Formula}

Let  ${\rm{trace}}\otimes {\rm{Id}}_{n\times
n}:\mathcal{M}_{mn}(\mathbb{C})\cong
\mathcal{M}_m(\mathbb{C})\otimes \mathcal{M}_n(\mathbb{C})\to
\mathbb{C}\otimes\mathcal{M}_n(\mathbb{C})\cong\mathcal{M}_n(\mathbb{C})$
be the operator defined as follows $$\big({\rm{trace}}\otimes
{\rm{Id}}_{n\times n}\big)(\sum_{i,j=1}^{m}E_{m}(i,j)\otimes
A_{i,j})=\sum_{i=1}^{m}A_{i,i},$$ where $E_m(i,j)$ is the $m\times
m$ matrix which is defined as follows
\begin{eqnarray*}
{(E_m{(i,j)})}_{k,l} =  \begin{cases} 0 &\mbox{if}
\,\, (k,l)\neq (i,j),\\ 1  & \mbox{if}\;(k,l)=(i,j). \end{cases}
\end{eqnarray*}
(An arbitrary
element $A$ in $\mathcal{M}_m(\mathbb{C})\otimes
\mathcal{M}_n(\mathbb{C})$ is of the form
$A=\sum_{i,j=1}^{m}E_{m}(i,j)\otimes A_{i,j}.$)

\begin{thm}
Let $0\to \mathcal{J}_{k-1}(E_f)\to\mathcal{J}_k(E_f)\to
{\mathcal{J}_k(E_f)}/{\mathcal{J}_{k-1}(E_f)}\to 0$ be an exact
sequence of jet bundles. Then we have $$\big({\rm{trace}}\otimes
{\rm{Id}}_{n\times n}\big)(\mathcal{K}_{\mathcal{J}_k(E_f)})-
\big({\rm{trace}}\otimes {\rm{Id}}_{n\times
n}\big)(\mathcal{K}_{\mathcal{J}_{k-1}(E_f)})=\mathcal{K}_{{\mathcal{J}_k(E_f)}/{\mathcal{J}_{k-1}(E_f)}}(z).$$
 \end{thm}

\begin{proof}

\begin{eqnarray*}
\lefteqn{\big({\rm{trace}}\otimes {\rm{Id}}_{n\times
n}\big)(\mathcal{K}_{\mathcal{J}_k(E_f)})-
\big({\rm{trace}}\otimes {\rm{Id}}_{n\times
n}\big)(\mathcal{K}_{\mathcal{J}_{k-1}(E_f)})}\\&&~~~~=\big(\det
\mathcal{J}_k(h)(z)\big)^{-1} \det
(\mathcal{J}_{k-1}h(z))h_{k}^{-1}(z)h_{k+1}(z)\\&&~~~~~~~~~~~~~
-\big(\det \mathcal{J}_{k-1}(h)(z)\big)^{-1} \det
(\mathcal{J}_{k-2}h(z))h_{k-1}^{-1}(z)h_{k}(z)\\
&&~~~~=\mathcal{K}_{{\mathcal{J}_k(E_f)}/{\mathcal{J}_{k-1}(E_f)}}(z).
\end{eqnarray*}
\smallskip
The last equality follows from \cite[page 244, Proposition
4.19]{cd}.
\end{proof}
\textsl{Acknowledgements:}{ Result of this paper contained in the
thesis titled  ``Infinitely Divisible Metrics, Curvature
Inequalities and Curvature Formulae" submitted at Indian Institute
of Science, Bangalore. The author would like to thank Professor
Gadaghar Misra and Dr. Cherian Varghese for their valuable
suggestions and numerous stimulating discussions relating to topic
of this paper.


\end{document}